\documentclass[12pt]{amsart}

\usepackage{amsfonts, amssymb, amsmath, latexsym,  mathtools, amscd,mathrsfs}
\usepackage{hyperref,graphicx}
\usepackage{apxproof}
\usepackage{xcolor}
\usepackage{url}
\usepackage{wrapfig}
\usepackage{subcaption,cleveref}
\usepackage{cite}

\newtheorem{Theorem}{Theorem}[section]

\newtheorem{Lemma}[Theorem]{Lemma}
\newtheorem{Corollary}[Theorem]{Corollary}

\newtheorem{Claim}{Claim}

\theoremstyle{remark}
\newtheorem{Remark}[Theorem]{Remark}
\newtheorem{Example}[Theorem]{Example}
\newtheorem{Definition}[Theorem]{Definition}

\newcommand{\CC}{{\mathbb C}}

\newcommand{\RR}{{\mathbb R}}
\newcommand{\TT}{{\mathbb T}}
\newcommand{\ZZ}{{\mathbb Z}}

\newcommand{\calA}{{\mathcal A}}

\newcommand{\calE}{{\mathcal E}}

\newcommand{\N}{{\mathcal N}}

\newcommand{\calV}{{\mathcal V}}

\newcommand{\A}{{\mathcal A}}
\newcommand{\Lc}{{\mathcal L}}
\newcommand{\calN}{{\mathcal N}}

\newcommand{\newt}[1]{\calN(#1)}

\newcommand{\defcolor}[1]{{\color{blue}#1}}
\newcommand{\demph}[1]{\defcolor{{\sl #1}}}

\definecolor{TAMU}{RGB}{140,0,0}
\definecolor{myblue}{RGB}{0,0,198}
\definecolor{myred}{RGB}{182,0,0}

\title{Rare Flat Bands for Periodic Graph Operators}

\author{Matthew Faust}
\address{Matthew Faust, Department of Mathematics, Texas A\&M University, College Station, TX 77843-3368, USA} 
\address{Current address: Department of Mathematics, Michigan State University, East Lansing, MI 48824, USA} \email{mfaust@msu.edu}
\urladdr{https://mattfaust.github.io/}
\author{Wencai Liu}
\address{Wencai Liu, Department of Mathematics,
         Texas A\&M University, College Station, Texas 77843,  USA}
\email{wencail@tamu.edu}
\urladdr{https://sites.google.com/view/wencail/home}

\subjclass[2020]{14M25, 47A75, 81Q10.}
\keywords{Dispersion Relation,  Periodic Graph, Periodic Graph Operators, Flat Bands. }
\begin{document}

\begin{abstract}
As a corollary of our main results, we prove that for any connected $\mathbb{Z}^d$-periodic graph, when edge weights and potentials are treated as variables, the corresponding periodic graph operators \emph{generically} (i.e., outside a proper algebraic subset of the variable space) do not have flat bands.
\end{abstract}

\maketitle 

\section{Introduction  and main result}

Flat bands are special spectral features of periodic operators, giving rise to eigenvalues with infinite multiplicity. This leads to a high density of states, a key factor in condensed matter physics. 
A prominent example is twisted bilayer graphene (TBG), where mathematical models have revealed the presence of flat bands with non-trivial topology \cite{becker2020mathematics,tarnopolsky2019origin,becker2022fine,becker2023degenerate,WL}. Similar phenomena have been rigorously established in twisted trilayer and multilayer graphene Hamiltonians, where distinct topological properties were identified \cite{becker2023chiral,yang2023flat}.

In this paper, we focus on $\mathbb{Z}^d$-periodic  graph operators, such as adjacency operators and discrete Schr\"odinger operators, and investigate conditions under which flat bands can (or cannot) occur. Flat bands of  periodic graph operators have attracted significant interest in the scientific community, e.g.,  see \cite{KFSH} and references therein for further discussion.

Periodic graph operators play a central role in both mathematics and physics. On the physical    side, they serve as  canonical models for crystalline networks (e.g. \cite{TopCry}), random walks in periodic environments, and quantum dynamics on lattices. By mimicking continuous periodic elliptic operators in a discrete setting, they provide a natural framework for studying Bloch waves, band structures, and gap phenomena in the tight-binding approximation. On the mathematical side, they exemplify the deep interplay among  group actions by $\mathbb{Z}^d$,  spectral theory, algebraic geometry, and combinatorics, enabling powerful techniques such as Floquet--Bloch analysis and commutative algebra.

The spectral theory of periodic graph operators has continued to be an active area of research. Recently, a wide range of topics has been explored, including   irreducibility results~\cite{lslmp20,shjst20, flm22, liu1, flm23, fg25}, extrema in the dispersion relation~\cite{FS,fk18,fk2,dksjmp20,bccm,bz},  isospectrality~\cite{liu2d,liujde,flmrp,liuborg}, quantum ergodicity~\cite{ms22, liu2022bloch}, quantum dynamics~\cite{adfs,dly,AScmp23}, the Bethe–Sommerfeld conjecture~\cite{ef,hj18,fh}, and embedded eigenvalues~\cite{liu1,shi1,kv06cmp,kvcpde20,lmt}. For a broader overview of these developments, see~\cite{ksurvey,kuchment2023analytic,liujmp22, shipmansottile,DF2}.

Despite the breadth of work on spectral theory for $\mathbb{Z}^d$-periodic operators, our understanding of flat bands  remains largely unknown. Certain special graphs---such as the Lieb, Kagome, and Dice lattices---are known to host flat bands under specific symmetry or hopping arrangements.  Other lattice geometries (e.g., Archimedean tilings~\cite{ktw,pt}) also admit robust  flat-band behaviors.  In~\cite{dksjmp20}, Do–Kuchment–Sottile established the non-degeneracy of critical points for a class of $\mathbb{Z}^2$-periodic graphs of two-atom fundamental domains with generic edge weights, a statement even stronger than the absence of flat bands. However, for general $\mathbb{Z}^d$-periodic graphs, systematic results are  limited.

Several interesting advances have nonetheless been made.
In \cite{sabri2023flat}, it was shown that for a connected periodic graph with positive weights and real potentials, the last band in the dispersion relation cannot be flat (see \cite{SK} for the case where all weights are 1). Furthermore, the authors demonstrated the absence of flat bands for fixed edge weights and generic potentials when the fundamental domain contains at most two vertices. Regarding the complete absence of flat bands, \cite{Higuchi} proved that the discrete Laplacian on a maximal abelian cover of a finite graph  $M$, where  $M$  has a 2-factor,  does not admit any flat bands.  Moreover, one could establish the nonexistence of flat bands for a broader class of periodic graph operators beyond those in \cite{flm22,fg25,liu1}, where the authors proved the irreducibility of  Bloch varieties—a stronger result than the absence of flat bands.
Nonetheless, these criteria remain far from exhaustive.

In this paper, as a corollary of our main results, we establish the following theorem:
\begin{Theorem}~\label{Thm:intro}
Assume that the underlying $\ZZ^d$-periodic graph is connected. When edge weights and potentials are treated as variables, we show that  generically (meaning outside a proper algebraic subset of the variable space), the corresponding periodic graph operators have no flat bands.
\end{Theorem}

For a more precise statement of Theorem \ref{Thm:intro}, refer to Theorem \ref{mainthm2}. Theorem \ref{Thm:intro} is an immediate consequence of our main result, Theorem \ref{Thm:Main}.

Theorem 1.1 addresses a longstanding question in the field. In particular, it may be viewed as a  discrete analog  of the following conjecture (see \cite[Conjecture~5.18]{ksurvey}), which states that a periodic elliptic second-order operator $H$ with sufficiently smooth coefficients cannot exhibit flat bands:

\begin{quote}
\cite[Conjecture 5.18]{ksurvey}.
\quad
\textit{A periodic elliptic second-order operator $H$ with sufficiently smooth coefficients does not admit any flat bands.}
\end{quote}

For the continuous Schr\"odinger operator $H = \Delta + V$, it is well known that no flat bands occur \cite{T73}, yet Conjecture~5.18 remains largely open for general periodic elliptic second-order operators\cite{ksurvey}. We refer readers to \cite{ksurvey} for further background and development on flat-bands of  such operators.

In the discrete setting, especially for discrete periodic Schr\"odinger operators with unit edge weights, certain families of graphs are already known to exhibit no flat bands; in fact, an even stronger statement holds in those cases, namely that their Bloch varieties are irreducible \cite{flm22,fg25,liu1}.

By proving that flat bands are absent under generic assumptions on the discrete graph and its weights, Theorem~1.1 effectively settles the discrete analogue of \cite[Conjecture~5.18]{ksurvey}, showing that flat bands are indeed rare in such periodic systems.

We would like to comment  that the generic assumption (i.e., a dense open set in the Zariski topology) is both natural and, in many cases, essential for studying the spectral theory of periodic graph operators. As mentioned earlier, certain well-known examples, such as the Lieb and Kagome lattices, can exhibit flat bands when appropriate potentials and edge weights are chosen.

The proof of Theorem~1.1 draws on techniques from combinatorics, algebraic geometry, and spectral analysis.  The starting point is that  a periodic graph operator admits a flat band if and only if the associated polynomial $ D(z,\lambda)$—which depends on the edge weights and potentials—has a linear factor in $\lambda$. The main steps of the proof are as follows:

\begin{itemize}
\item[Step 1] We establish some essential properties of the Newton polytope  $N(D(z,\lambda))$, which play a key role in the subsequent proofs. See Section~\ref{Sec:NoCancel}.
    \item[Step 2] We first handle the special case where the Newton polytope $N(D(z,\lambda))$ is a vertical line segment.  
See Section~\ref{Sec:Vertical}. 
    \item [Step 3] Using combinatorial arguments, we show that every proper vertical facial polynomial of the Newton polytope $N(D(z,\lambda))$ must be independent of at least one potential variable.  See Section~\ref{Sec:Comb}.
    
    \item [Step 4] 
    We show that if the periodic graph operator has a flat band, then some proper face of the Newton polytope $N(D(z,\lambda))$ necessarily contains a vertical line segment.  Together with Step 3, we demonstrate that after deleting the orbit of a certain vertex, the resulting periodic graph operator shares the same flat band as the original. See Section~\ref{Sec:LinearFactors}.

    \item [Step 5] Finally, we complete the proof by induction using a resultant-based argument. See Section~\ref{Sec:Resultant} for the case where the   quotient graph has only cut edges and Section~\ref{Sec:Result} for the general case.

\end{itemize}
Beyond settling a longstanding problem, we hope this paper will appeal to researchers in \emph{mathematical physics}, \emph{combinatorics}  and \emph{algebraic geometry}, showcasing how ideas from these diverse fields come together to address spectral problems in periodic graph operators.

\section{Background and Main Results }~\label{Sec:MainRes}
\subsection{Periodic Graphs and Discrete Periodic Operators}
A graph $\Gamma = (\mathcal V, \mathcal E) $ is said to be  $\mathbb{Z}^d $-periodic if:

	1.	There exists a free action of  $\mathbb{Z}^d$  on  $\Gamma$, meaning no nonzero element of  $\mathbb{Z}^d $ fixes any vertex.
    
	2.	The action of $\mathbb{Z}^d $ is invariant, i.e., for every edge  $(u, v) \in \mathcal E $ and every $a \in \mathbb{Z}^d$, we have $( u+a,  v+a) \in \mathcal E $.
    
	3.	The action is cocompact, meaning that the quotient graph  $
    \Gamma/ \mathbb{Z}^d $ is a finite graph.

Given a collection of vertices $U \subset \calV(\Gamma)$ and a vector $a \in \ZZ^d$, we write $U+a$ for the collection of vertices $\cup_{u \in U} u + a$. For any $\ZZ^d$-periodic graph $\Gamma$, there exists a finite set $W \subset \calV(\Gamma)$ such that $\cup_{a \in \ZZ^d} W+a = \calV(\Gamma)$ and $|W| = |\calV(\Gamma)/\ZZ^d|$. We call such a set $W \subset \calV(\Gamma)$ a \demph{(vertex) fundamental domain}.  See Figure \ref{fig:LiebGraph} for an example. 
 Because
the action is free, each $u\in \calV(\Gamma)$ has a unique representation as $u=v+a
$ with $ v\in W$  and $a\in\ZZ^d$.

When there is no ambiguity,  we will use $n$ in place of $|W|$, and label the vertices of $W$ by the elements of $\demph{[n] := \{ 1,\dots, n\}}$.

A \demph{$\ZZ^d$-periodic edge labeling} $E : \calE(\Gamma) \to \CC$ is a periodic and symmetric function:  for any $(u,v) \in \calE(\Gamma)$, $$E((u,v))=E((v,u))$$ and for all $a \in \ZZ^d$, $$E((u,v)) = E((u+a,v+a)).$$
We will often call the value $E((u,v))$ the \demph{label} of the edge $(u,v)$. 

An edge labeling on $\Gamma$ gives rise to a \demph{labeled adjacency operator $A_E$}. $A_E$ acts on a function $f$ on $\calV(\Gamma)$ as follows, \[A_E(f)(u) = \sum_{v:(u,v) \in \calE(\Gamma)} E((u,v))f(v), u \in \calV(\Gamma).\] 

A \demph{$\ZZ^d$-periodic potential} $V : \calV(\Gamma) \to \CC$ is a periodic function on $\calV(\Gamma)$ such that $V(u) = V(u+a)$ for all $u \in \calV(\Gamma)$ and $a \in \ZZ^d$. We call a pair of functions $(V,E)$ a \demph{labeling} of $\Gamma$. 

 As the label of any self-loops of $\Gamma$ can be absorbed into the potential and the labels of multiple edges between two vertices of $\Gamma$ can be combined into a single label, from now on,  we assume that $\Gamma$ has no self-loops or multiple edges.

Given a labeling $(V,E)$ of $\Gamma$, the operator $\demph{\Lc} := V + A_E$, acts on a function $f$ on $\calV(\Gamma)$ as follows:

\[\Lc(f)(u) = V(u) f(u) + \sum_{v:(u,v) \in \calE(\Gamma)} E((u,v))f(v), u \in \calV(\Gamma).\]

We will use $\demph{\CC^{V,E}}$ to denote the space of all complex valued labelings of $\Gamma$. 

\subsection{Floquet Theory}
We aim to study the spectrum of a discrete periodic operator $\Lc$ acting on $\ell^2(\calV(\Gamma))$, the Hilbert space of square summable functions on $\calV(\Gamma)$. 
Fix a vertex fundamental domain $W$ of $\Gamma$, and let $$\TT^d := \{(u_1,\dots, u_d) \in \CC^d \mid  |u_i| = 1 \text { for each } i  \}.$$ The Floquet transform $\mathscr{F}$ of a function $f$ on $\calV(\Gamma)$ is given by: \[ f(u) \mapsto \hat{f}(z,u) = \sum_{a \in \ZZ^d} f(u+a)z^{-a}.\]

Notice that $\hat{f}(z,u+a) = z^a \hat{f}(z,u)$. One can see that $\hat{f}(z,u)$ is just the Fourier transform of $f$ restricted to the orbit $u + \ZZ^d \subseteq \calV(\Gamma)$. If $f \in \ell^2(\calV(\Gamma))$,  then by Parseval's identity,  we have
\begin{equation}\label{gmar22}
    \sum_{u \in W} \int_{\TT^d} \vert \hat{f}(z,u) \vert^2  dz=\sum_{v\in \mathcal V(\Gamma)} |f(v)|^2.
\end{equation}

Thus, we can view $\hat{f}(z,\cdot) = (\hat{f}(z,1), \hat{f}(z,2), \dots, \hat{f}(z,n))^T$ as an element of the Hilbert space $L^2(\TT^d,\CC^n)$. By \eqref{gmar22}, $\mathscr{F}$ is a unitary operator; so, $\Lc$ and $\mathscr{F} \Lc \mathscr{F}^*$ are unitarily equivalent.

Direct calculation implies that, for $u \in \calV(\Gamma)$,
\begin{equation}\label{gflo1}
    \mathscr{F} \Lc \mathscr{F}^*(\hat{f})(z,u) = V(u) \hat{f}(z,u) + \sum_{ w \in W,a\in\ZZ^d\atop{(u,w+a) \in \calE(\Gamma)}} E((u,w+a))z^a \hat{f}(z,w).
\end{equation}

For each $z \in \TT^d$, let $L(z)$ be a $|W|\times |W|$ matrix   (with rows and columns indexed by the vertices of $W$) given by the following: for $u,v \in W$,
\begin{equation}\label{gflo2}
    L(z)_{u,v} = \delta_{u,v}V(u)+\sum_{a\in\ZZ^d:(u,v+a)\in \calE(\Gamma)}E((u,v+a))z^{a},
\end{equation}
 where $\delta$ is the Kronecker delta function ($\delta_{u,v} = 1$ when $u = v$ and is $0$ otherwise). 

By \eqref{gflo1} and \eqref{gflo2}, $ \mathscr{F} \Lc \mathscr{F}^*$ is the direct integral of $L(z)$:
\begin{equation}~\label{eq:decomposition} \mathscr{F} \Lc \mathscr{F}^* = \int_{\TT^d}^{\oplus} L(z) \ dz.\end{equation}

For many of our arguments, we will treat the potential and edge labels as indeterminates. When we use elements of $[n] = W$ for vertices of the fundamental domain, we will write $\demph{v_i}$ for the indeterminate corresponding to the potential of vertex $i$ (that is, $V(i) = v_i$), and $\demph{e_{(i,j),a}}$ for the indeterminate corresponding to the label of the edge $(i,j+a)$, where $i,j \in W, a \in \ZZ^d$ (that is, $E((i,j+a)) = e_{(i,j),a}$). We will use $v$ to represent the vector of the potential indeterminates, and $e$ to represent the vector of edge label indeterminates. By the symmetry of the edge label (namely $E((u,v))=E((v,u))$), we have 
\begin{equation}\label{gfeb251}
    e_{(i,j),a} = e_{(j,i),-a}.
\end{equation}

Under above notations, we have that \eqref{gflo2} can be expressed as:
 \begin{equation}\label{eq:floquetmatrix}
    L(z)_{i,j} = \delta_{i,j}v_i + \sum_{a\in\ZZ^d}e_{(i,j),a} z^a.
\end{equation}

 We remark that each entry of $L(z)$ is a Laurent polynomial in $z$, and \eqref{gfeb251} implies
 \begin{equation}\label{gfeb252}
   L(z) = L(z^{-1})^T.
\end{equation}
 \smallskip
\begin{Example}~\label{Ex:InducedGraph}
    Let $\Gamma$ be the Lieb lattice  (Figure~\ref{fig:LiebGraph}):
    
\begin{figure}[ht]
    \centering
\includegraphics[scale=2]{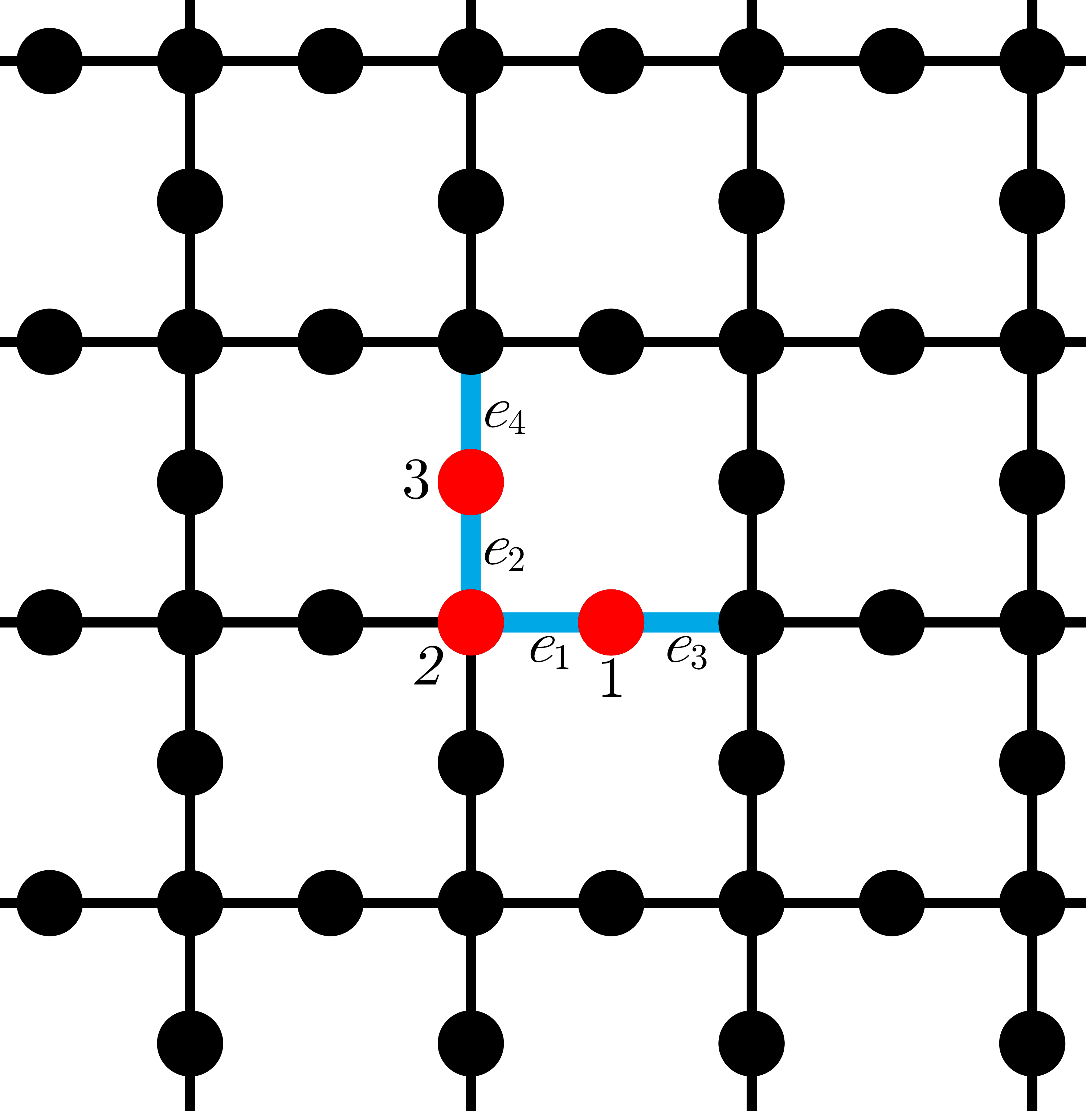}
    \caption{The $\ZZ^2$-periodic Lieb lattice.}
    \label{fig:LiebGraph}
\end{figure}
The three red dots/vertices form a fundamental domain. Label the three vertices by 1, 2, 3 as
in  Figure~\ref{fig:LiebGraph}. The Lieb lattice has four distinct edges up to periodicity (modulo $\ZZ^2$):  $$e_1:=E((1,2)) = e_{(1,2),(0,0)} , $$
$$e_2:=E((2,3)) = e_{(2,3),(0,0)} ,$$
$$e_3:=E((1, 2 + (1,0))) = e_{(1,2),(1,0)} ,$$ 
$$ \text{and } e_4:=E((3, 2 + (0,1))) = e_{(3,2),(0,1)}.$$
The potential is given by  $V(i) = v_i$ for $i = 1,2,3$.
    Under the above labeling, the Floquet matrix $L(z)$ has the following expression: 
    \[\text{\makebox[\displaywidth][c]{$L(z)=\begin{pmatrix}
        v_1 & e_1 +e_3z_1& 0 \\
        e_1 +e_3z_1^{-1}  &  v_2 & e_2+e_4z_2^{-1}  \\
        0 & e_2+e_4z_2 & v_3
    \end{pmatrix}.$ }} \text{\makebox[0pt][r]{\hspace{\displaywidth}$\diamond$}}\]
    \end{Example}
 
\subsection{Flat bands and  precise formulation of main result}

Let $\Lc$ be a discrete periodic operator associated to a $\ZZ^d$-periodic graph $\Gamma$. We call the characteristic polynomial of $L(z)$ the \demph{dispersion polynomial} and denote it by \begin{equation}~\label{eq:dispersionPoly}
    \demph{D(z,\lambda)} := \det(L(z)-\lambda I).
\end{equation} The \demph{dispersion relation} (or Bloch variety) is given by the set of zeros of $D(z,\lambda)$ in $(\CC^*)^d \times \CC$. 

Note that $L(z)$ (also $\Lc$) depends on $v$ and $e$. Sometimes, we write $L(v,e,z)$ and $D(v,e,z,\lambda)$ when we want to emphasize the dependence of $v$ and $e$.

\begin{Definition}~\label{def:flatband}
    The operator $\Lc$ has a \demph{flat band} if there exists a constant $\lambda_0$   such that 
$\lambda - \lambda_0$  is a factor of $D(z,\lambda)$.
 In this case, we say that $\Lc$ has a flat band at the energy level $\lambda_0$. \hfill $\diamond$
\end{Definition}

\begin{Remark}
    In applications of mathematical physics, the edge labels are positive and the potentials are real.
 When the potential and edge labeling are real valued, \eqref{gfeb252} implies that for $z \in \TT^d$, $L(z)$ is Hermitian and, therefore, has only real eigenvalues. 
For each $z \in \TT^d$, let $\lambda_1(z) \leq \dots \leq \lambda_j(z) \leq \dots \leq \lambda_n(z)$ be the eigenvalues of $L(z)$. We call $\lambda_j(z)$ the \demph{$j$th spectral band function} of $\Lc$.
In this case, $\Lc$ has a flat band at the energy level $\lambda_0$  
if and only if there exists   $j \in [n]$ such that $\lambda_j(z) = \lambda_0$ for all $z \in \TT^d$. \hfill $\diamond$
\end{Remark}

We are now ready to precisely state  Theorem \ref{Thm:intro}.

\begin{Theorem}\label{mainthm2}
    Suppose that $\Gamma$ is connected. There is a proper algebraic set $S$ of $\CC^{V,E}$ such that for all fixed $(v,e)\in \CC^{V,E}\backslash S$,
     $D(v,e,z,\lambda)$ has no linear factor in $\lambda$ (aka $\Lc$ has no flat bands).
    
\end{Theorem}

\section{ Basics of  algebra and  notations}~\label{Sec:Basics}
\subsection{Some Algebra and Discrete Geometry}
Denote by $\CC[x_1^{\pm}, \dots, x_k^{\pm}]$$=\CC[x^{\pm}]$
all Laurent polynomials.
Write $f = \sum_{a \in \ZZ^k} c_a x^a$.
 We will often use the notation $\demph{[x^a]}f = c_a$ to denote the \demph{coefficient} of the term with monomial $x^a$ in $f$.  The \demph{support} of $f$, denoted $\demph{\A(f)} \subset \ZZ^k$, is the finite collection of $a \in \ZZ^k$ such that $c_a \neq 0$.

The \demph{Newton polytope of} $f$, denoted $\demph{\newt{f}} \subseteq \RR^k$, is the convex hull of $\A(f)$. In particular, \[\newt{f} := \{ b_1 a_1 + \dots + b_{|\A(f)|} a_{|\A(f)|} \mid a_i \in \calA(f),  0 \leq b_i \leq 1, \sum_{i=1}^{|\A(f)|} b_i = 1\}.\]

The \demph{Minkowski sum} of two polytopes $P$ and $Q$ is given by
\[ P+Q := \{p+q \mid p \in P, q \in Q\}.\]

If a (Laurent) polynomial $f$ factors, that is, there exist (Laurent) polynomials $g$ and $h$ such that $f = gh$, then the Newton polytope of $f$ must be the Minkowski sum of the  Newton polytopes of its factors: \[\N(f) = \N(g) + \N(h).\] 

Let $w \in \ZZ^{k}$ and let $m = \min_{a \in \A(f)} w \cdot a$, where $w \cdot a = \sum_{i=1}^{k} w_i a_i$. We often call $w \cdot a$ the \demph{weight} of the vector $a$ with respect to $w$, or the weight of the term $c_a x^a$ with respect to $w$. The \demph{face} $F = \{a \in \newt{f} \mid a \cdot w = \min_{a \in \newt{f}} a \cdot w= \min_{a \in \mathcal A(f)} a \cdot w\}$ of $\newt{f}$ identified by $w$ is the collection of points of $\newt{f}$ minimized by $w$.    Write  $f_w$ or $f_F$ for the corresponding \demph{facial polynomial}:
 \[\demph{f_w} := \sum_{a \in F \cap \A(f)} c_{a} x^a.\]

Notice that if $f$ is a polynomial and $f = pq$, then \begin{equation}~\label{eq:facial} f_w = (pq)_w = p_w q_w. \end{equation}

If $f = \sum_{j=1}^k f^j$ and $\mathcal{A}(f^j)\subseteq \mathcal{A} (f),j=1,2,\cdots,k$, then \begin{equation}~\label{eq:facialsum} f_w = (\sum_{j=1}^k f^j_w )_w. \end{equation}

\subsection{Support, Components, Induced Graphs, Induced Operators, and Simplified Quotient Graphs}
Let $\Gamma$ be a $\ZZ^d$-periodic graph with a fundamental domain $W$.

\begin{Definition}
    The \demph{support} of a collection of vertices $U \subseteq W$, denoted $\demph{\calA(U)}$, is the collection of $a \in \ZZ^d$ such that there exists an edge in $\calE(\Gamma)$ between some vertex of $U$ and some vertex of $U+a$. That is,
\[ \calA(U) := \{ a \in \ZZ^d \mid \text{ there exist }  u\text{ and } v \text{ in } U \text{ such that } (u, v + a) \in \calE(\Gamma) \}. \]

If $\calA(U) = \{ \emptyset\}$ or $\{(0,\dots,0) \}$, then we say that $U$ has \demph{support $0$}. \hfill $\diamond$
\end{Definition}

\begin{Definition} We call $U \subseteq W$ a \demph{component} of $W$ if there does not exist $u \in U$ and $v \in W \smallsetminus U$ such that there is an edge of $\calE(\Gamma)$ with end points in both $u + \ZZ^d$ and $v + \ZZ^d$. That is, $U$ is a component of $W$ if the set of edges \[ \{(u,v+a) \in \calE(\Gamma) \mid u \in U, v \in W \smallsetminus U, a \in \ZZ^d \} \] is empty. \hfill $\diamond$
\end{Definition}

\begin{Definition}
For $U \subseteq W$, the \demph{induced} $\ZZ^d$-periodic graph $\demph{\Gamma_U}$ is given by \[\calV(\Gamma_U) := \{u+a \mid u \in U, a \in \ZZ^d \} \] 
and
\[\  \calE(\Gamma_U) := \{ (u+a,v+b) \in \calE(\Gamma) \mid u,v \in U, a,b \in \ZZ^d \},\] 
and is a subgraph of $\Gamma$. \hfill $\diamond$
\end{Definition}

 Let $U \subseteq W$. Given a labeling $(V,E)$ of $\Gamma$, we can discuss restricting the labeling to $\Gamma_U$.

\begin{Definition}
    Define the \demph{induced potential} $V|_U: \calV(\Gamma_U) \to \CC$ and the \demph{induced edge labeling} $E|_U: \calE(\Gamma_U) \to \CC$ to be the natural restrictions of $V$ and $E$ to the vertices and edges of $\Gamma_U \subseteq \Gamma$, respectively. 
    
    The \demph{induced labeling} of $(V,E)$ on $\Gamma_U$ is given by $\demph{(V|_U, E|_U)}$. \hfill $\diamond$
\end{Definition}

With an induced labeling, we may define an \demph{induced   periodic operator}.

\begin{Definition}
Suppose \(\Lc\) is an operator on \(\Gamma\) with labeling \((V, E)\). Define \(\demph{\Lc|_U}\) as the \demph{induced discrete periodic operator} on \(\Gamma_U\) with labeling \((V|_U, E|_U)\). We let \(\demph{L|_U(z)}\) be its Floquet matrix and \(\demph{D|_U(z,\lambda)}:= \det(L|_U(z) - \lambda I)\) be its dispersion polynomial. \hfill \(\diamond\)
\end{Definition}

\smallskip

\begin{Lemma}~\label{fact:componentzero} Let $\Gamma$ be a $\ZZ^d$-periodic graph with a fundamental domain $W$. If $W$ has a nonempty component $U$ of support $0$, then $D(z,\lambda)$ has a linear factor in $\lambda$.
\end{Lemma}
\begin{proof}
 By definition of $U$ having support $0$, the $\ZZ^d$-periodic graph $\Gamma_U$ has a fundamental domain $U$ such that there are no edges between vertices of $U$ and $U+a$ in $\calE(\Gamma_U)$ for any $a \neq 0 \in \ZZ^d$. By \eqref{eq:floquetmatrix},  entries of $L|_U(z)$ are  independent of $z$, and thus $D|_U(z,\lambda) = \det(L|_U(z) - \lambda I)$ is a univariate polynomial in $\lambda$. Therefore, $D|_U(z,\lambda)$ has $|U|$ linear factors in $\lambda$. Furthermore, as $U$ is a component of $W$, $\Gamma = \Gamma_U \sqcup \Gamma_{W \smallsetminus U}$, and so, possibly after reordering the vertices of $W$,

 \begin{equation} L(z) =  \begin{pmatrix}
    L|_U(z) & \bf{0} \\
    \bf{0} &  L|_{W\smallsetminus U}(z)
\end{pmatrix}. \end{equation}
It follows that the linear factors of  $\det(L|_U(z) - \lambda I)$  must divide $D(z,\lambda)$.
\end{proof}

\begin{Definition}
    Let $\demph{S(\Gamma / \ZZ^d)}$ be the finite simple graph with vertices $$\calV(S(\Gamma / \ZZ^d)) := W$$ and edges $$ \calE(S(\Gamma / \ZZ^d)) := \{ (u,v)\in W\times W\mid u \neq v  \text{ and } (u,v+a) \in \calE(\Gamma) \text{ for some  } a  \in \ZZ^d \}.$$ 
    We call $S(\Gamma / \ZZ^d)$ the \demph{simplified quotient graph} of $\Gamma$.  \hfill $\diamond$
\end{Definition}

The simplified quotient graph is the simple graph counterpart of the quotient graph $\Gamma / \ZZ^d$ (see \cite{TopCry}).

\subsection{Generic Support and the Generic Newton Polytope}

\begin{Definition}
    The collection of all algebraic sets of $\CC^{V,E}$ induces a topology whose closed sets are the algebraic sets.   This induced topology is the \demph{Zariski topology}.  \hfill $\diamond$
\end{Definition}
Notably, if $Z (\neq \CC^{V,E})$ is a proper  algebraic set, then the Zariski open set $\CC^d \smallsetminus Z$ is a dense open set.
\begin{Definition}
A property is called \demph{generic} if it is true on some nonempty Zariski open subset of $\CC^{V,E}$. Any point $(V,E)$ in such a subset is a \demph{generic labeling}. \hfill $\diamond$
\end{Definition}

A polynomial $f \in \CC[z^{\pm},\lambda,v,e]$ has the following expression: 
 \begin{equation}~\label{eq:gendisppoly}
    f(z,\lambda)=  f(v,e,z,\lambda)= \sum_{(a,b)\in\ZZ^{d+1} \atop{c_{a,b}\neq 0}} c_{a,b}(v,e) z^a \lambda^b.
\end{equation} 
Let $Var(c_{a,b}(v,e))$ be the set of zeros of $c_{a,b}(v,e)$ in $\CC^{V,E}$. It follows that the set $$\CC^{V,E} \smallsetminus (\bigcup Var(c_{a,b}(v,e)))$$ is a nonempty Zariski open set in $\CC^{V,E}$. Therefore, generically (with respect to $(v,e)$),  the support of $f$ (as a function of $z$ and $\lambda$) is fixed.  This motivates the following definitions:

\begin{Definition}
 
         When $v,e$ are taken to be vectors of indeterminates, define the \demph{generic support} of a polynomial $f \in
         \CC[z^{\pm},\lambda,v,e]$ as the support of $f(z,\lambda) \in \CC[z^{\pm},\lambda]$  for a generic choice of $(v,e) \in \CC^{V,E}$. Denote by $A(f)$ 
         the \demph{generic support} of $f$. \hfill $\diamond$

\end{Definition}

\begin{Definition}
 
         The \demph{generic Newton polytope} (denoted by $N(f)$) of $f=f(z,\lambda)=f(v,e,z,\lambda)$  is the Newton polytope $\calN(f(z,\lambda))$ for a generic labeling $(v,e) \in  \CC^{V,E}$.  \hfill $\diamond$

\end{Definition}
 
It is clear that for any fixed $(v,e) \in \CC^{V,E}$, $\calN(f) \subseteq N(f) $.

Define  $\demph{N}:=N(D(z,\lambda))$ and     refer to it as the generic Newton polytope of  $\Gamma$.

\subsection{Vertical Faces and Permutations}
\begin{Definition}
    A face $F$ of a polytope $P \subset \RR^{d+1}$ is called a \demph{vertical face} if there exist two points $(a,b), (a,c) \in F$ such that $a \in \ZZ^d$ and $b, c \in \ZZ$, where $b \neq c$.
   
    A vertical face that is also an edge (that is, a $1$-dimensional polytope) is called a \demph{vertical segment}.
     \hfill  $\diamond$
\end{Definition}

For example, if $f \in \CC[z^{\pm},\lambda]$, and there is a face $F$ of $\N(f(z,\lambda)) \subset \RR^{d+1}$ such that $f_F$ has two terms $z^a\lambda^b$ and $z^a \lambda^c$ with nonzero coefficients for some integers $b\neq c$ and some $a \in \ZZ^d$, then $F$ is a vertical face.  

Notice that a dispersion polynomial $D(z,\lambda)$ must always be such that $$[\lambda^n]D(z,\lambda) = \pm 1.$$ It follows that, for a fixed labeling, the Newton polytope $\calN(D(z,\lambda)) \subset \RR^{d+1}$ is a vertical segment if and only if $D(z,\lambda)$ can be written as a univariate polynomial in $\lambda$.

Let $\demph{S_n}$ denote the collection of permutations of $[n]$. We will often express a permutation $\sigma \in S_n$ as a product of finitely many disjoint cycles $\eta_1,\dots, \eta_k$. Let \begin{equation}~\label{eq:perm}
    \demph{\sigma D(z,\lambda)} := \prod_{i=1}^n (L(z) - \lambda I)_{i,\sigma(i)}.
\end{equation}
Similarly, if $\eta$ is a cycle of a permutation $\sigma$ of length $p$, that is, $\eta = (j_1\dots j_p)$ where each $j_i \in [n]$, then we write \begin{equation}~\label{eq:cycle} \demph{\eta D(z,\lambda)} := \prod_{i=1}^p (L(z) - \lambda I)_{j_i,\eta(j_i)}.\end{equation}

The definition of determinant states that \begin{equation} \label{defdet}
D(z,\lambda) = \sum_{\sigma \in S_n} sgn(\sigma) \sigma D(z,\lambda).\end{equation}

\subsection{Summary of Notation and Reformulation of Main Theorem}

Our main goal is to prove the following theorem.

\begin{Theorem}~\label{Thm:Main} Let $\Gamma$ be a $\ZZ^d$-periodic graph. Then for a generic labeling, the dispersion polynomial $D(z,\lambda)$ has a linear factor in $\lambda$ if and only if $\Gamma$ has a fundamental domain $W$ with a nonempty component of support $0$. 
\end{Theorem}
\begin{Remark}
   By Lemma~\ref{fact:componentzero}, Theorem~\ref{Thm:intro} (aka Theorem~\ref{mainthm2}) follows directly from Theorem~\ref{Thm:Main}.
\end{Remark}
We present all the notation and conventions here that will be repeatedly used throughout the proofs.

\begin{enumerate}
    \item $\A(f)$ the support of $f$. $A(f)$ the generic support of the polynomial $f$.
    \item $f_w$ the facial polynomial identified by a vector $w$. 
    \item $\N(f)$ the Newton polytope of $f$.  $N(f)$ the generic Newton polytope of the polynomial $f$. 
    \item $\Gamma = (\calV(\Gamma),\calE(\Gamma))$ the $\ZZ^d$-periodic graph with vertices $\calV(\Gamma)$ and edges $\calE(\Gamma)$.
    \item $W := [n]$ a fundamental domain of $\Gamma$.
    \item $S(\Gamma / \ZZ^d)$ the simplified quotient graph of $\Gamma$.
    \item $\A(U)$   the support of $U \subset \calV(\Gamma)$.
    \item $(V,E)$ the labeling of some $\ZZ^d$-periodic graph $\Gamma$.
    \item $v,e$ denote when $V$ and $E$ are taken to be vectors of indeterminates.
    \item $\Lc$  the  periodic graph operator.
    \item $L(z)=L(v,e,z)$ the Floquet matrix.
    \item $D(z,\lambda)=D(v,e,z,\lambda)$ the dispersion polynomial. 
    \item $N=N(D(z,\lambda))$ the generic Newton polytope
    of $\Gamma$.
    \item In the following proofs, unless otherwise specified, we always assume a generic labeling.
\end{enumerate}
\medskip

\section{Newton polytopes}~\label{Sec:NoCancel}

\begin{Theorem}~\label{Lem:Inner2}
For all $\sigma \in S_n$,
$A(\sigma D(z,\lambda)) \subseteq A(D(z,\lambda))  \subset \RR^{d+1}$.
\end{Theorem}

To prove Theorem~\ref{Lem:Inner2}, we need the following two lemmas. In these lemmas, we will consider the support of $\sigma D(v,e,z,\lambda)$ as a polynomial in $\CC[z^{\pm},\lambda, v ,e]$, and completely describe the collection of $\tau \in S_n$ such that $\calA(\sigma D(v,e,z,\lambda)) \cap \calA(\tau D(v,e,z,\lambda)) \subset \ZZ^{d+1} \times \ZZ^{V,E}$ may be nonempty. 
\medskip

\begin{Lemma}~\label{Lem:InnerHelper2}
Let $\Gamma$ be a $\ZZ^d$-periodic graph with a (vertex) fundamental domain $W = [n]$. Assume that  there exists  $r=(r_1,r_2,r_3,r_4)$   with $r_1\in\ZZ^d$, $r_2\in\ZZ$, and $(r_3,r_4)\in \ZZ^{V,E}$ such that $r\in \calA(\sigma D(v,e,z,\lambda)) \cap \calA(\tau D(v,e,z,\lambda)) $. Then the following statements hold
\begin{itemize}
    \item[1.] 
    \begin{equation}\label{gfeb265}
    [z^{r_1}\lambda^{r_2} v^{r_3} e^{r_4}] \tau D(v,e,z,\lambda)= [z^{r_1}\lambda^{r_2} v^{r_3} e^{r_4}]\sigma D(v,e,z,\lambda).
\end{equation}
    \item[2.] Write the permutation  $\sigma \in S_n$  as the   product of disjoint cycles $\sigma = \eta_1,\dots, \eta_k$. Then 
 for each $i \in [k]$, either $\eta_i$ or $(\eta_i)^{-1}$ is a cycle of $\tau$.
\end{itemize}
\end{Lemma}
\begin{proof}
Let $r=(r_1,r_2,r_3,r_4)$ with $r_1\in\ZZ^d$, $r_2\in\ZZ$, and $(r_3,r_4)\in \ZZ^{V,E}$ be such that 
\begin{equation}\label{gfeb261}
    r\in \calA(\sigma D(v,e,z,\lambda)) \cap \calA(\tau D(v,e,z,\lambda)) .
\end{equation}
Decompose  $\sigma \in S_n$ as a product of disjoint of cycles $\eta_1, \dots, \eta_k$ for some $k \geq 1$, such that $\eta_i = (j_{i,1},\dots, j_{i,l_i})$ and the first $p \geq 0$ of the $\eta_i$ are the $1$-cycles $(j_{i,1}) = (j_i)$ of this product.

Recalling \eqref{eq:perm} and \eqref{eq:cycle}, \begin{equation} \label{gmar35}\sigma D(v,e,z,\lambda) = \prod_{i=1}^k \eta_i D(v.e,z,\lambda) = \prod_{i=1}^n ((L(v,e,z) - \lambda I)_{i,\sigma(i)}).\end{equation}

For $i = p+1,\dots, k$, as $\eta_i$ has no fixed points, $\eta_i D(v,e,z,\lambda)$ is a product of off-diagonal entries of $L(v,e,z)$. It follows from \eqref{eq:floquetmatrix} and \eqref{eq:cycle} that a term of $\eta_i D(v,e,z,\lambda)$ must be of the form \begin{equation}\label{gfeb263}
     e_{(j_{i,1}, j_{i,2}), a_{i,1}} \dots e_{(j_{i,l_{i}-1}, j_{i,l_i}), a_{i,l_i-1}} e_{(j_{i,l_{i}}, j_{i,1}), a_{i,l_i}  } \prod_{s=1}^{l_i} z^{a_{i,s}}, \end{equation}
for some $a_i =(a_{i,1},\dots, a_{i,l_i})\in (\ZZ^d)^{l_i}$. 

 As our $e_{(i,j),a}$ are independent, apart from the restriction that $e_{(i,j),a} = e_{(j,i),-a}$, we must have that $e_{(i,j),a}$ only appears in $L(v,e,z)_{i,j}$ and $L(v,e,z)_{j,i}$. Therefore, by \eqref{gfeb261},\eqref{gmar35} and \eqref{gfeb263}, we have that
  \begin{equation} [e_{(j_{i,1}, j_{i,2}), a_{i,1}} \dots e_{(j_{i,l_{i}-1}, j_{i,l_i}), a_{i,l_i-1}} e_{(j_{i,l_{i}}, j_{i,1}), a_{i,l_i}}]e^{r_4} \neq 0,\end{equation}
  and hence
 $\tau$ must contain $\eta_i$ or $(\eta_i)^{-1}$. 

 By interchanging $\tau$ and $ \sigma$, we conclude that for any non 1-cycle $\gamma$ appearing in the product decomposition of $\tau$, 
$\sigma$ must contain $\gamma$ or $\gamma^{-1}$.  
We finish the proof by noting that the remaining terms (1-cycles) only  involve  diagonal entries and that the nontrivial coefficient for the potential and edge variables (variable $\lambda$) in each entry $L(z)_{i,j}$  is 1  (-1).  
\end{proof}

\smallskip

\smallskip
\begin{Lemma}~\label{Lem:InnerHelper}
Let $\Gamma$ be a $\ZZ^d$-periodic graph with a (vertex) fundamental domain $W = [n]$.  Then we have for all permutations $\sigma\in S_n$,
\begin{equation}\label{gfeb266}
    \calA(\sigma D(v,e,z,\lambda)) \subseteq \calA(D(v,e,z,\lambda) \subset \ZZ^{d+1} \times \ZZ^{V,E}.
\end{equation}

\end{Lemma}
\begin{proof}
Fix an arbitrary permutation $\sigma\in S_n$. Clearly, \eqref{gfeb266} holds if $\calA(\sigma D(v,e,z,\lambda))$ is empty. So, we will assume $\calA(\sigma D(v,e,z,\lambda))$ is nonempty and 
choose an arbitrary $r=(r_1,r_2,r_3,r_4) \in \calA(\sigma D(v,e,z,\lambda)) \subset \ZZ^{d+1} \times \ZZ^{V,E}$, where  $r_1 \in \ZZ^d$, $r_2 \in \ZZ$, $(r_3,r_4) \in \ZZ^{V,E}$. Our goal is to show that  $$r \in \calA(D(v,e,z,\lambda)) \subset \ZZ^{d+1} \times \ZZ^{V,E}.$$ 

As $D(v,e,z,\lambda) = \sum_{\tau \in S_n} sgn(\tau) \tau D(v,e,z,\lambda)$, we have \begin{align} \label{gfeb2612}\begin{split}[z^{r_1}\lambda^{r_2} v^{r_3} e^{r_4}] D(v,e,z,\lambda) = [z^{r_1}\lambda^{r_2} v^{r_3} e^{r_4}] \sum_{\tau \in S_n} sgn(\tau) \tau D(v,e,z,\lambda) \newline \\= \sum_{\tau \in S_n} sgn(\tau) [z^{r_1}\lambda^{r_2} v^{r_3} e^{r_4}] \tau D(v,e,z,\lambda).\end{split}\end{align}
If $[z^{r_1}\lambda^{r_2} v^{r_3} e^{r_4}] \tau D(v,e,z,\lambda)\neq 0$, 
by Lemma~\ref{Lem:InnerHelper2}, we have that 
\begin{equation}\label{gfeb2610}
    [z^{r_1}\lambda^{r_2} v^{r_3} e^{r_4}] \tau D(v,e,z,\lambda)= [z^{r_1}\lambda^{r_2} v^{r_3} e^{r_4}]\sigma D(v,e,z,\lambda),
\end{equation}
and 
\begin{equation}\label{gfeb2611}
    sgn(\tau)=sgn(\sigma).
\end{equation}

Now, Lemma \ref{Lem:InnerHelper} follows from \eqref{gfeb2612}-\eqref{gfeb2611} and the fact that $r=(r_1,r_2,r_3,r_4) \in \calA(\sigma D(v,e,z,\lambda))$.
\end{proof}

\smallskip

Recall that $A(f) \subset \ZZ^{d+1}$ and $N(f) \subset \RR^{d+1}$ are used to denote the support and Newton polytope of a polynomial $f \in \CC[z^{\pm},\lambda, v,e]$ for a generic choice of labeling.

\begin{proof}[\bf Proof of Theorem~\ref{Lem:Inner2}]
Assume $\sigma \in S_n$. Let $r = (r_1,r_2) \in A(\sigma D(z,\lambda)) \subset \ZZ^{d+1}$, where $r_1 \in \ZZ^d$ and $r_2 \in \ZZ$.

As $r \in A(\sigma  D(z,\lambda)) \subset \ZZ^{d+1}$, $[z^{r_1} \lambda^{r_2}] \sigma D(z,\lambda)$ must be a nonzero polynomial in $\CC[v,e]$. By Lemma~\ref{Lem:InnerHelper},  $[z^{r_1} \lambda^{r_2}] D(z,\lambda)$ is nonzero as a polynomial in $\CC[v,e]$. We conclude that $r \in A(D(z,\lambda))$.
\end{proof}
\smallskip

The following corollary follows immediately from Theorem~\ref{Lem:Inner2}.
\begin{Corollary}~\label{Lem:Inner}
     Let $w \in \ZZ^{d+1} \neq 0$ be such that $N(D_w(z,\lambda))$ is a proper face of $N(D(z,\lambda))$, and let $m = \min_{a \in A(D(z,\lambda))} w \cdot a$. 
     
    Fix any labeling. For any $\sigma \in S_n$, $\sigma D(z,\lambda)$ cannot contain a term with weight less than $m$ with respect to $w$.
\end{Corollary}

\section{The Case of Vertical Segments}~\label{Sec:Vertical}
\begin{Theorem}~\label{Lem:GenPolytope} Let $\Gamma$ be a $\ZZ^d$-periodic graph. The generic Newton polytope of $\Gamma$ is a vertical segment if and only if $\Gamma$ has a fundamental domain of support $0$. 
\end{Theorem}

The proof of Theorem~\ref{Lem:GenPolytope} relies on the following two lemmas.
\begin{Lemma}~\label{Lem:Help2GenPolytope} Let $\Gamma$ be a $\ZZ^d$-periodic graph. Suppose that the generic Newton polytope $N$ of $\Gamma$ is a vertical segment, then \begin{enumerate} 
\item $N$ is the vertical segment  connecting $(0,\dots,0)$ and $(0,\dots, 0, n)$, 
\item for every $\sigma \in S_n$, $\sigma D(z,\lambda)$ must be a univariate polynomial in $\lambda$, 
\item and for every cycle $\eta$, $\eta D(z,\lambda)$ must be a univariate polynomial in $\lambda$.\end{enumerate}
\end{Lemma}
\begin{proof}

Notice that \begin{equation}\label{eq:constant} [v_1 v_2 \cdots v_n] D(v,e,z,\lambda) = 1.\end{equation} It follows that for a generic labeling, $D(z,\lambda)$ has a nonzero constant term.

As $D(z,\lambda)$ has a nonzero constant term and a nonzero term $\pm \lambda^n$ for a generic labeling, we see that $N$ contains the vertices $(0,\dots,0)$ and $(0,\dots, 0, n)$. Thus, $N$ is a vertical segment if and only if there does not exist any $(a,b) \in N$ such that $a \neq 0 \in \ZZ^{d}$ and $b \in \ZZ$. By~\eqref{eq:dispersionPoly}, $n$ is the largest possible exponent of $\lambda$ and $0$ is the smallest possible exponent of $\lambda$ that can appear in $D(z,\lambda)$. Thus, if $N$ is a vertical segment, it is exactly the vertical segment connecting $(0,\dots,0)$ and $(0,\dots, 0, n)$.

By Theorem~\ref{Lem:Inner2}, we have $N(\sigma D(z,\lambda)) \subseteq N$, and so $N(\sigma D(z,\lambda))$ is contained in the vertical segment with endpoints $(0,\dots,0)$ and $(0,\dots, 0, n)$. Upon fixing a labeling, as $\N(\sigma D(z,\lambda)) \subseteq N(\sigma D(z,\lambda))$, it follows that $\sigma D(z,\lambda) \in \CC[\lambda]$.

Finally, note that if $\eta = (j_1,\dots, j_p)$ (where $p \geq 1$) is a cycle, then let $\sigma$ be the permuation such that $\sigma_i = \eta(i)$ for $i \in \{j_1,\dots, j_p\}$ and $\sigma(i) = i$ otherwise. For a generic potential \begin{equation} p(z,\lambda) := \prod_{i \in [n] \smallsetminus \{j_1,\dots, j_p\}} (L(z)-\lambda I)_{i,i}\end{equation} has a nonzero constant term. As $\sigma D(z,\lambda) = p(z,\lambda) \eta D(z,\lambda)$, we have that \begin{equation} N(\eta D(z,\lambda)) \subset N(\sigma D(z,\lambda)).\end{equation} As $N(\eta D(z,\lambda))$ is contained in the vertical segment with endpoints $(0,\dots,0)$ and $(0,\dots, 0, n)$, we conclude that     $\eta D(z,\lambda) \in \CC[\lambda]$. 
\end{proof}

\medskip
\begin{Lemma}~\label{Lem:HelpGenPolytope} Let $\Gamma$ be a $\ZZ^d$-periodic graph with an $n$ vertex fundamental domain $W$ and suppose that the generic Newton polytope $N$ of $\Gamma$ is a vertical segment. If there exists $U \subset W$, where $|U| = |W|-1$, such that $U$ is a fundamental domain of support $0$ of $\Gamma_U$, then $\Gamma$ has a fundamental domain of support $0$. 
\end{Lemma}
\begin{proof}
Assume that our labeling is generic, and assume $W = [n]$. Without loss of generality, assume $U = \{2,\dots, n \}$ is a support $0$ fundamental domain of $\Gamma_U$. 

As $U$ has support 0, $L|_U(z)$ has only constant entries. As the generic Newton polytope $N$ of $\Gamma$ is a vertical segment, $L(z)_{1,1}$ must be a constant. Otherwise for  the cycle  $\eta=(1)$,   $\eta D(z,\lambda)$ is not a univariate polynomial in $\lambda$, contradicting Lemma~\ref{Lem:Help2GenPolytope}.  We are left to consider the matrix entries $L(z)_{1,i}$ and $L(z)_{i,1}$ for $i \in [n]$ with $i\geq 2$. 

By Lemma~\ref{Lem:Help2GenPolytope}, as for each cycle $\eta = (1 \  i)$, $i \in [n]$, $\eta D(z,\lambda)$ is in $\CC[\lambda]$, each $L(z)_{1,i}$ is either $0$ or a single term in $\CC[z^{\pm}]$. By \eqref{gfeb252}, if $L(z)_{j,1} = c_2 z^{a}$, then $L(z)_{1,j} = c_2 z^{-a}$.

Remark that there are only finitely many distinct vectors $A_1, \dots, A_l \in \ZZ^d$ such that $L(z)_{1,i} = c z^{-A_j}$ for some $i \in [n]$ and some constant $c \neq 0$. Consider the simplified quotient graph of $\Gamma_U$, $S(\Gamma_{U}/\ZZ^d)$. Notice that if $i$ and $i'$ are such that $L(z)_{(1,i)} = c z^{-A_j}$ and $L(z)_{(1,i')} = c' z^{-A_{j'}}$ for some $j \neq j'$ and $c,c' \neq 0$, then $i$ and $i'$ cannot be   connected in $S(\Gamma_{U}/\ZZ^d)$. Indeed, if $i$ and $i'$ are   connected by $i, u_1,\dots, u_s, i'$, then there exists a cycle $\eta = (1 \ i \  u_1  \ \dots \  u_s \ i')$ (noting that the submatrix $L|_{U}(z)$ is a matrix of all constants) such that $\eta D(z,\lambda) \not \in \CC[\lambda]$, contradicting Lemma~\ref{Lem:Help2GenPolytope}.  

Let us partition $[n] \smallsetminus \{1 \}$ into sets $I_1,\dots, I_s$, where each $I_j$ is a connected component of $S(\Gamma_{U}/\ZZ^d)$, $j=1,2,\dots, s$.
 By the previous paragraph, we have that for each $I_j$,  there exists $A_j$ such that  $L(z)_{1,k}=c_k z^{-A_j}$  for all $k\in I_j$  (we allow $c_k=0$). Now consider the new fundamental domain $W'$ given by $\{1\}\cup (\cup_{j \in [s]} I_j+A_j)$.

It is easy to check that  $W'$ is a fundamental domain of $\Gamma$ of support $0$. 
\end{proof}

\begin{proof}[\bf Proof of Theorem~\ref{Lem:GenPolytope}]

Note that one direction is trivial (follows from the proof of Lemma~\ref{fact:componentzero}).

We are left to show that if the generic Newton polytope of $\Gamma$ is a vertical segment, then $\Gamma$ has a fundamental domain $W$ of support $0$. 

We will prove this by induction.

Let $\Gamma$ be a $\ZZ^d$-periodic graph, let $W = [n]$ be some vertex fundamental domain of $\Gamma$, and assume that our labeling is generic. 

Let $n=1$. In this case, it is obvious that $D(z,\lambda) =L(z)_{1,1}-\lambda$ is a polynomial in $\lambda$ if and only if $\Gamma$ has a fundamental domain of support $0$.

We continue by induction on $n$. Suppose that when $n <k$, if $N$ is a vertical segment in $\lambda$, then $\Gamma$ has a fundamental domain of support $0$. 

Consider the case when $n = k$, and assume that the generic Newton polytope of $\Gamma$ is a vertical segment. Let $U = [k] \smallsetminus \{1 \}$. By  \eqref{eq:perm} and \eqref{defdet}, one has that

\begin{equation}\label{gm11}
\sum_{\sigma \in S_k \mid \sigma(1)=1} \sigma D(z,\lambda) = (L(z)_{1,1}-\lambda) D|_U(z,\lambda).\end{equation}

As $(L(z)_{1,1}-\lambda)$ has a nonzero constant ($v_1$), by \eqref{gm11}, for a generic labeling, $$N(D|_U(z,\lambda)) \subseteq N(D(z,\lambda)) = N.$$ Therefore, the generic Newton polytope of $\Gamma_U$ (given by $N(D|_U(z,\lambda))$)   is a vertical segment. From the induction hypothesis, it follows that $\Gamma_U$ has a fundamental domain $U'$ of support $0$. After noting that $W' = \{1 \} \cup U'$ is a fundamental domain of $\Gamma$, it follows from Lemma~\ref{Lem:HelpGenPolytope} that there exists a fundamental domain of support $0$ of $\Gamma$. 
\end{proof}

\section{The Combinatorics of a Proper   Face}\label{Sec:Comb}

\begin{Theorem}~\label{Lem:PotIndep} 
If $F = N_w\subset \RR^{d+1}$ is a proper vertical face, then there exists  $i \in [n]$ such that  $D_w(z,\lambda)$   is  independent of $v_i$.
\end{Theorem}

For Theorem~\ref{Lem:PotIndep}, we will make use of some notions from graph theory. We now give the graph-theoretic background necessary for our proofs. Those who already have a basic background in graph theory can jump to Lemma~\ref{lem:propvert}.

Given a finite graph $G=(U,T)$ with vertices $U$ and edges $T$, a collection of vertices $U' \subset U$ is an \demph{independent set}, if for all pairs of vertices $u_1,u_2 \in U'$, the edge $(u_1,u_2)$ is not an element of $T$. The graph $G$ is \demph{bipartite} if there exist two independent sets $U_1, U_2 \subset U$ such that $U_1 \cup U_2 = U$ and $U_1 \cap U_2 = \{ \emptyset \}$. The graph $G$ is a \demph{multigraph} if we allow $T$ to contain multiple edges between the same two vertices. Given a multigraph $G$, the \demph{degree} of a vertex $u \in U$ is the number of edges connected to it. The (multi)graph $G$ is \demph{$k$-regular} if every vertex has degree $k$. A collection of edges $P \subset T$ is a \demph{perfect matching} (or $1$-factor) if the graph $G' = (U,P)$ is $1$-regular. The graph $G$ is \demph{weighted} if we assign a numerical value to each edge of $S$. 

Let $G = (U,T)$ be a $k$-regular ($k>0$) bipartite multigraph such that $U$ can be partitioned into independent sets $U_1$ and $U_2$ such that $|U_1| = |U_2| = s$. \demph{Hall's marriage theorem} states that $G$ has a perfect matching $P$. If we remove the edges $P$ from $G$, we are left with a $k-1$-regular bipartite multigraph (if $k-1 >0$, then we can again apply Hall's marriage theorem).

\begin{Lemma}\label{lem:propvert}
    Let $N$ be the generic Newton polytope of $\Gamma$. If there exists a proper vertical face $F \subsetneq N$ identified by a vector $w \in \ZZ^{d+1}$, then 
    \begin{enumerate}
        \item $w = (w_1,\dots, w_d, 0) \neq (0,\dots, 0)$,
        \item $F$ cannot contain the point $(0,\dots, 0)$,
        \item and $w \cdot a = m < 0$ for all $a \in F$.
    \end{enumerate} 
\end{Lemma}
\begin{proof}
    As $F$ is a proper vertical face, $F \subsetneq N$. It follows that $N$ cannot be a vertical segment.

Let $w \in\ZZ^{d+1}$  be the vector identified $F$, namely
 \begin{equation}~\label{eq:face} F = \{a \in N \mid w \cdot a = \min_{b \in N} w \cdot b \}.\end{equation}
   Clearly  $w \neq (0,\dots, 0)$ (otherwise, $F = N$).   Let $m \in \RR$ be such that $w \cdot a = m$ for all $a \in F$.

   As $F$ is a vertical face, it must contain two points $(a,b)$ and $(a,c)$, where $a \in \ZZ^d$, $b,c \in \ZZ$, $b \neq c$, and $w \cdot (a,b)= w \cdot (a,c) = m$. Thus, $w_{d+1}b = w_{d+1}c$, and so $w = (w_1,\dots, w_d, 0)$.
The following fact is well known. For the reader’s convenience, we provide a proof after completing the proof of Lemma \ref{lem:propvert}.

\smallskip
    \begin{Claim}~\label{claim:1}If a point $u$ lies in $F$, then any segment contained in $N$ that contains the point $u$ in its interior must be contained in $F$.
    \end{Claim}

    Suppose now that $(0,\dots,0)$ lies on the vertical face $F$. As $w = (w_1,\dots, w_d,0)$, by~\eqref{eq:face}, $F$ must contain the segment connecting $(0,\dots, 0)$ and $(0,\dots, 0,n)$. Since $N$ is not a vertical segment, there exist points $(a,b) \in N$ such that $a \neq (0,\dots, 0) \in \RR^d$ and $b \in [0,n]$.
    Let $(a,b) \in N$ be such a point.
    By~\eqref{gfeb252}, $(-a,b) \in N$. As $N$ is convex, $N$ contains the segment $P_{a,b} = [(a,b), (-a,b)]$, which contains the point $(0,\dots,0,b)$ in its interior. As $(0,\dots, 0,b) \in F$, by Claim~\ref{claim:1}, one has that $P_{a,b} \subseteq F$ , and so $(a,b) \in F$. As our choice of $(a,b)$ was arbitrary, all the points of $N$ must be contained in $F$, but then $F = N$; contradicting that $F$ is proper. It follows that $(0,\dots,0)$ cannot be a point of $F$. 

    As $(0,\dots, 0) \not\in F$ and $(0,\dots,0)\in N$ (by \eqref{eq:constant}),  we conclude that
    $$m =  \min_{b \in N} w \cdot b < w \cdot (0,\dots, 0) = 0.$$
\end{proof}
\begin{proof}[\bf Proof of Claim 1]
         Suppose that $F$ is a face of a polytope $N$. If $u \in F$ and $u$ is also in the interior of a segment $[p_1,p_2] \subset N$,  our goal is to prove that $[p_1,p_2] \subseteq F$.  The segment $[p_1,p_2]$ is given by $s(t) = p_1 (1-t) + p_2 t$, where $t \in [0,1]$, and there exists $t_0 \in (0,1)$ so that $s(t_0) = u$. Let $\omega(t) = w \cdot s(t)$, a linear function of $t$. By~\eqref{eq:face}, $\omega(t)$ achieves its minimum at $t_0$. As $\omega$ is linear in $t$, 
         $\omega(t) =  m$ for all $t \in [0,1]$. Therefore,   $[p_1,p_2] \subseteq F$.
    \end{proof}

\begin{proof}[\bf Proof of Theorem  ~\ref{Lem:PotIndep}]

Suppose that $w \in \ZZ^{d+1}$ identifies a proper vertical face $F$ of $N$. As $F$ is a vertical face, by Lemma~\ref{lem:propvert} we have that $w = (w_1,\dots, w_d, 0)$ and 
\begin{equation}\label{gfeb2619}
    \demph{m :=} \min_{a \in A(D(z,\lambda))} w \cdot a < 0.
\end{equation} 
By  ~\eqref{eq:facial}, ~\eqref{eq:facialsum}, \eqref{eq:perm},  \eqref{defdet} and Theorem \ref{Lem:Inner2}, we have that \begin{equation}~\label{eq:facialexpansion} D_w(z,\lambda) = (\sum_{\sigma \in S_n} sgn(\sigma)   (\sigma D(z,\lambda))_w)_w=(\sum_{\sigma \in S_n} sgn(\sigma) \prod_{j=1}^{n} ((L(z)-\lambda I)_{j,\sigma(j)})_w)_w.\end{equation}

Suppose that $D_w(z,\lambda)$ is dependent on $v_i$ for each $i \in [n]$. Then by \eqref{eq:facialexpansion}, we have that
there exists a collection of, not necessarily distinct, permutations $\sigma_1, \dots, \sigma_n \in S_n$   such that for all  $i\in[n]$, $\sigma_i(i) = i$,
\begin{equation}\label{gfeb2618} 
m = \min_{a \in A(\sigma_i D(z,\lambda))} w \cdot a,
\end{equation}
and
\begin{equation}\label{gmar33}
    [v_i](L(z)_{i,i} - \lambda)_w =1.
\end{equation}

Let $G = (U,T)$ be a weighted multigraph with $2n$ vertices $$U = \{u_1, \dots, u_n, t_1, \dots, t_n\},$$ and edge set \begin{equation}\label{eq:edges} T = \{ (u_j, t_{\sigma_i(j)}) \mid i \in [n], j \in [n] \} .\end{equation}  By \eqref{eq:edges}, $G$ is biparite with independent sets $\{u_1,\dots, u_n\}$ and $\{t_1, \dots, t_n \}$,  and  $G$ is $n$-regular. For each edge between  $u_i$ and $t_{j}$, define its weight as 
\begin{equation}\label{gmar38}
    \omega(u_i,t_j) := \min_{a \in A((L(z) - \lambda I)_{i,j})}  w \cdot a.
\end{equation}

By \eqref{eq:facial} and \eqref{eq:perm}, for each $i$ we have
\begin{equation}\label{gfeb2617}
    \sigma_i D(z,\lambda)_w = \prod_{j=1}^n ((L(z) - \lambda I)_{j,\sigma_i(j)})_w.
\end{equation}

By \eqref{gfeb2618} and \eqref{gfeb2617}, 
\begin{equation}\label{eq:edgesum} m   = \sum_{j=1}^n \min_{a \in A((L(z) - \lambda I)_{j,\sigma_i(j)})} w \cdot a  = \sum_{j=1}^n \omega(u_j,t_{\sigma_i(j)}). 
\end{equation}

By \eqref{eq:edges} and \eqref{eq:edgesum}, one has that \begin{equation}\label{eq:Gedgesum}
    \sum_{(u,t) \in T} \omega(u,t) = \sum_{i = 1}^n \sum_{j=1}^n \omega(u_j,t_{\sigma_i(j)}) = nm.
\end{equation}

Recalling that $\sigma(i) = i$ for each $i \in [n]$, from   \eqref{eq:edges} it follows that $T$ contains the edges $(u_j,t_j)$ in $G$ for each $j \in [n]$. Notice that \[P = \{(u_1,t_1), (u_2,t_2), \dots, (u_n,t_n)\}\] is a perfect matching of $G$. 

By \eqref{gmar33}, one has that
$$(0,\dots,0) \in A((L(z)_{i,i} - \lambda)_w),$$ and  hence $\omega(u_i,t_i) = 0$ for each $i \in [n]$.

After removing the edges of $P$ from $G$, we are left with an $(n-1)$-regular bipartite multigraph $G' = (U,T \smallsetminus P)$. As $P$ only has edges of weight $0$, by \eqref{eq:Gedgesum}, $$\sum_{(u,t) \in (T \smallsetminus P) } \omega(u,t) = nm.$$

As $G'$ is $(n-1)$-regular, by Hall's marriage theorem, we can consecutively remove $n-1$ more perfect matchings $P_1,\dots, P_{n-1}$ from $G'$.
As there are only $n-1$ perfect matchings and the total weight of the edges of $G'$ is $nm$ (noting that $m<0$ by \eqref{gfeb2619}), by the pigeon hole principle, there exists $k\in \{1,2,\dots,n-1\}$ such that 
\begin{equation}\label{gmar31}
    \sum_{(u,t) \in P_k} \omega(u,t) < m.
\end{equation} 

Since $P_k$ is a perfect matching, 
it is easy to see that there exists a  permutation $\tau\in S_n$ such that 
\begin{equation}~\label{eq:pmatch} P_k = \{(u_1, t_{\tau(1)}),\dots, (u_n,t_{\tau(n)})\}. \hfill
\end{equation}

By \eqref{gmar38}, \eqref{gmar31} and \eqref{eq:pmatch}, we have that $$\sum_{j=1}^n \omega(u_j, t_{\tau(j)}) = \min_{a \in A(\tau_k D(z,\lambda))} w\cdot a< m,$$ contradicting Corollary~\ref{Lem:Inner}. 
\end{proof}

\section{The Existence of Vertical Faces, and The Inheritance of Linear Factors}~\label{Sec:LinearFactors}
\begin{Theorem}~\label{Thm:2}
    Let $\Gamma$ be a $\ZZ^d$-periodic graph with no fundamental domain of support $0$.   
    
    If for a generic labeling, $D(z,\lambda)$  has a linear factor in $\lambda$,
    then for a generic labeling, there exists an $i \in W$ such that $D|_{W \smallsetminus i}(z,\lambda)$  and $D(z,\lambda)$ share a linear factor in $\lambda$.

\end{Theorem}

\begin{Lemma}~\label{Lem:VertFace}
Let $\Gamma$ be a $\ZZ^d$-periodic graph with no fundamental domain of support $0$.

If for a generic labeling, $D(z,\lambda)$ has a linear factor in $\lambda$, then $N$ must have a proper face given by a vertical segment. 
In particular, there exists $w\in\ZZ^{d+1}$ with the last coordinate 0 such that 
 $D_{w}(z,\lambda) =  z^a p(\lambda)  $  with $a$ nonzero and $p(\lambda)$ nonzero, and 
 all linear factors of $D(z,\lambda)$ divide $p(\lambda)$.
\end{Lemma}
\begin{proof}
Fix a generic labeling. As $D(z,\lambda)$ has a linear factor, we have that there exists some constant $\lambda_0$ such that $D(z,\lambda) = (\lambda -\lambda_0)g(z,\lambda),$ where $N(g(z,\lambda))$ is not a vertical segment by Theorem~\ref{Lem:GenPolytope}. It follows that $N = \newt{\lambda-\lambda_0} + N(g(z,\lambda))$ for a nonzero $\lambda_0$ (if $\lambda_0$ is $0$, this would contradict \eqref{eq:constant}). If we fix a generic value $\lambda_1 \in \CC$, then we can consider the polytope $N(g(z,\lambda_1)) \subset \RR^d$.  As $N(g(z,\lambda))$ is not a vertical segment, there exists a nonzero vertex of $N(g(z,\lambda_1))$. Fix $a$ to be such a vertex. As $a$ is a vertex (i.e., a $0$-dimensional face), there exists a vector $w' = (w'_1,\dots, w'_d) \in \ZZ^d$ such that \[ \{a \} = \{ c \in N(g(z,\lambda_1)) \mid c \cdot w' = \min_{c' \in \newt{g(z,\lambda_1)}} c' \cdot w'\} \ \ \  (= N(g_{w'}(z,\lambda_1))).\] 

Let $w = (w'_1,\dots, w'_d, 0) \in \ZZ^{d+1}$. 
Noting that $g_{w'}(z,\lambda_1)  = s z^a$ for some constant $s$, it follows that $g_{w}(z,\lambda) = z^a p(\lambda)$ for some nonzero univariate polynomial $p(\lambda)$. 

As the last coordinate of $w$ is $0$, $(\lambda-\lambda_0)_{w} = \lambda -\lambda_0$. By ~\eqref{eq:facial}, we conclude that $$D_{w}(z,\lambda) = z^a p(\lambda) (\lambda-\lambda_0)$$ with $a$ nonzero, and so $N(D_{w}(z,\lambda))$ is a vertical segment and a proper subset of $N$.  
\end{proof}

\begin{proof}[\bf Proof of Theorem \ref{Thm:2}]
 By Lemma~\ref{Lem:VertFace}, 
    there exists a proper face $F$ of $N(D(z,\lambda))$  identified by  a nonzero vector $w = (w_1,\dots,w_d, 0) \in \ZZ^{d+1}$ (that is, $F = N(D_w(z,\lambda))$) such that $D_w(z,\lambda)=z^{a} p(\lambda)$ and all linear factors of  $D(z,\lambda)$ divide $p(\lambda)$.
    By Theorem~\ref{Lem:PotIndep}, up to relabeling the vertices, we can assume that $D_w(z,\lambda)$ is independent of $v_1$ (that is, $[v_1] D_w(z,\lambda) = 0$).

  Note that $p(\lambda)$ has finitely many linear factors.
 After fixing a generic choice of $v_2, \dots, v_n$ and a generic edge labeling, it follows that for a generic choice of $v_1$ (i.e., all but finitely many values of $v_1$), at least one linear factor of $p(\lambda)$ (denote it by $\lambda-\lambda_0$) must divide $D(z, \lambda)$. Therefore, there exists an $\omega \in \CC$ such that $(\lambda-\lambda_0) | D(z,\lambda)$ for $v_1 = \pm \omega$.
  As \begin{equation} D(z,\lambda) = v_1 \det(L|_U(z) - \lambda I) + o.t,\end{equation} where $o.t$ is a polynomial independent of $v_1$. By summing $D(z,\lambda)$ when $v_1 = \omega$  and $v_1 = -\omega$, we obtain that $2 o.t$. It follows that $(\lambda-\lambda_0)$ must divide $o.t$, and therefore it must also divide $\det(L|_U(z) - \lambda I)$. 
This implies that for a generic  labeling, $D|_U(z,\lambda)$
 and $D(z,\lambda)$ have a common linear factor in $\lambda$.
 \end{proof}

\section{On the Resultant   of Graphs with only cut edges}~\label{Sec:Resultant}
We now introduce assumptions and notation for Theorem~\ref{Lem:CutGraph}.

Given two polynomials $f(\lambda) = \sum_{i=0}^s a_i \lambda^i$ and $g(\lambda)=\sum_{i=0}^t b_i \lambda^i$ of degree $s$ and $t$ respectively, the \demph{Sylvester matrix} is the $(s+t) \times (s+t)$ matrix
\[{\small \begin{pmatrix}
    a_0 & a_1 & a_2 & \dots & a_{s-1} & a_s & 0 & 0 & \dots &0 \\
    0 & a_0 & a_1 & \dots & a_{s-2} & a_{s-1} & a_s & 0 & \dots & 0 \\
    \vdots & \vdots & \ddots & \ddots & \ddots& \ddots& \ddots& \ddots & \ddots & \vdots \\
    0 & 0 & 0 & \dots & a_{0} & a_{1} & a_2 & a_3 & \dots & a_s \\
    b_0 & b_1 & b_2 & \dots & b_{t-1} & b_t & 0 & 0 & \dots &0 \\
    0 & b_0 & b_1 & \dots & b_{t-2} & b_{t-1} & b_t & 0 & \dots & 0 \\
    \vdots & \vdots & \ddots & \ddots & \ddots& \ddots& \ddots& \ddots & \ddots & \vdots \\
    0 & 0 & 0 & \dots & b_{0} & b_{1} & b_2 & b_3 & \dots & b_t \\
\end{pmatrix}},\]
where there are $t$ rows involving $a$'s and $s$ rows involving $b$'s. The determinant of the Sylvestor matrix is the (classical) \demph{resultant}, denoted $\demph{Res(f,g)}$. The resultant $Res(f,g)$ is a polynomial in the coefficients of $f$ and $g$ such that $Res(f,g) = 0$ if and only if $f$ and $g$ share a common factor.  

Suppose we have a $\ZZ^d$-periodic graph $\Gamma$ with fundamental domain $W = [n]$, $n>1$, such that there is a $U \subset W$, where $|U|=|W|-1$ and $U$ has support $0$. Furthermore, assume that the simplified quotient graph $H = S(\Gamma / \ZZ^d)$ of $\Gamma$ is connected and has only cut edges, that is, removing any edge from $H$ will disconnect $H$.  

\begin{Theorem}\label{Lem:CutGraph} Let $\Gamma$, $U$, $W$, and $H$ be as above. If all edge labels are nonzero, then $Res(D(z,\lambda)$, $D|_U(z,\lambda)) \neq 0$ for a generic choice of potential. 
\end{Theorem}
\begin{proof}
As $U$ has support $0$, we may write $D|_U(z,\lambda)$ as  $D|_U(\lambda)$. Therefore, $Res(D(z,\lambda)$, $D|_U(\lambda)) \neq 0$ if and only if 
$D(z,\lambda)$ and  $D|_U(\lambda)$ do not share common linear factor in $\lambda$.

Without loss of generality, we  assume that $W \smallsetminus U = \{1 \}$. 

 Write $ D(z,\lambda) $ as
 \begin{equation}
     D(z,\lambda) = (L(z)_{1,1} - \lambda) D|_U(\lambda) + P(z,\lambda).
 \end{equation} 

As $D|_U(\lambda)$ clearly divides $(L(z)_{1,1} - \lambda) D|_U(\lambda)$, we have that \begin{equation} Res(D(z,\lambda), D|_U(\lambda)) \neq 0\text{ if and only if }Res(P(z,\lambda), D|_U(\lambda)) \neq 0.\end{equation} We proceed by induction on the number vertices in $W$.

If $n = 2$, then $D|_U(\lambda)$ is a degree one polynomial in $\lambda$ and $$P(z,\lambda) = -L(z)_{1,2} L(z)_{2,1}$$ is a nonzero polynomial (otherwise, the graph would not be connected). It follows trivially that $P(z,\lambda)$ and $ D|_U(\lambda)$ share no factors, and so $Res(P(z,\lambda), D|_U(\lambda)) \neq 0$.

Suppose that  for all $n <k$, Theorem \ref{Lem:CutGraph} holds and consider the case when $n=k$.

Let $J$ be the graph obtained after removing the vertex $1$ from $H$.  This graph has $r$, where $r$ is the degree of vertex $1$ in $H$, connected components, $J_1, J_2,\dots, J_r$, each of which is a graph such that every edge is a cut edge. Moreover,  for each connected component $J_l$, there is a single edge between $1$ and some vertex $i_l$ of $J_l$. 

Up to relabeling the vertices,   we may write $L(z)$ as 

\[{\small L(z) = \begin{pmatrix}
    L(z)_{1,1} & \begin{matrix} L(z)_{1,i_1} & 0 & \dots 0 \end{matrix} & \dots  &\begin{matrix} L(z)_{1,i_r} & 0 & \dots 0 \end{matrix} \\ 
    
    \begin{matrix} L(z)_{i_1,1} \\ 0  \\ \vdots \\ 0 \end{matrix} & L|_{\calV(J_1)}(z) & \begin{matrix} 0\\ 0 \\ \vdots \\ 0 \end{matrix} & \bf{0}  \\
    
    \vdots & \begin{matrix} 0 & 0 & \dots 0 \end{matrix}  & \ddots & \\ 
    
    \begin{matrix} L(z)_{i_r,1} \\ 0  \\ \vdots \\ 0 \end{matrix} & \bf{0}
 & \cdots & L|_{\calV(J_r)}(z) \\
\end{pmatrix}}.\]

Thus, we see that \begin{equation}~\label{eq:8.3} P(z,\lambda) = -\sum_{j=1}^r \left( L(z)_{1,i_j} L(z)_{i_j,1} D|_{\calV(J_j) \smallsetminus \{i_j\}}(\lambda) \prod_{s=1, s \neq j}^r  D|_{\calV(J_s)}(\lambda)\right)\end{equation}
and \begin{equation}~\label{eq:8.4} D|_U(\lambda) =  \prod_{s=1}^r  D|_{\calV(J_s)}(\lambda).\end{equation}

We note that, as the vertex sets of $\calV(J_j)$ and $\calV(J_{j'})$ are disjoint whenever $j \neq j'$, the linear factors of $ D|_{\calV(J_j)}(\lambda) = \det(L|_{\calV(J_j)}(z)-\lambda I)$ and $D|_{\calV(J_{j'})}(\lambda) = \det(L|_{\calV(J_{j'})}(z)-\lambda I)$ do not agree for a generic choice of potential.

By \eqref{eq:8.3} and \eqref{eq:8.4}, $P(z,\lambda)$ and $D|_U(\lambda)$ share a common factor for a generic potential if and only if, for some $j$,
\begin{equation}\label{gfeb281}
    Res( D|_{\calV(J_j)}(\lambda), L(z)_{1,i_j} L(z)_{i_j,1} D|_{\calV(J_j) \smallsetminus \{i_j\}}(\lambda)) = 0.
\end{equation}

Note that for each $j$, $L(z)_{1,i_j} L(z)_{i_j,1}$ is nonzero.

{\bf Case 1}: $J_j$ is a single vertex

In this case, one has that \begin{align} \label{gfeb283}\begin{split} Res( D|_{\calV(J_j)}(\lambda), L(z)_{1,i_j} L(z)_{i_j,1} D|_{\calV(J_j) \smallsetminus \{i_j\}}(\lambda)) \\ = Res( D|_{\calV(J_j)}(\lambda), L(z)_{1,i_j} L(z)_{i_j,1}) \neq 0.\end{split} \end{align}

{\bf Case 2}:  $J_j$ has at least two vertices

Notice that   $J_j$  only has cut edges. As $\calV(J_j) \subset U$, $\Gamma_{\calV(J_j)}$ is a periodic graph with simplified quotient graph $J_j$. Furthermore, $\Gamma_{\calV(J_j)}$ also has a fundamental domain that is of support $0$, and has less than $k$ vertices. Finally, $U' = \calV(J_j) \smallsetminus \{i_j\}$ is a subset of $\calV(J_j)$ with support $0$. It follows from the induction hypothesis that  for a generic potential, 
\begin{equation*} 
    Res( D|_{\calV(J_j)}(\lambda), D|_{\calV(J_j) \smallsetminus \{i_j\}}(\lambda)) \neq 0,
\end{equation*}
and hence
\begin{equation}\label{gfeb282}
    Res( D|_{\calV(J_j)}(\lambda), L(z)_{1,i_j} L(z)_{i_j,1}D|_{\calV(J_j) \smallsetminus \{i_j\}}(\lambda)) \neq 0.
\end{equation}

By \eqref{gfeb281}, \eqref{gfeb283} and \eqref{gfeb282},
we conclude that  when $n=k$, $Res(D(z,\lambda), D|_U(\lambda)) \neq 0.$ By induction, we finish the proof.
\end{proof}

\section{\bf Proof of Theorem~\ref{Thm:Main}}~\label{Sec:Result}
\begin{proof}
The direction that having a component of support $0$ implies the existence of flat bands  follows from Lemma~\ref{fact:componentzero}.

We now prove the remaining direction of Theorem~\ref{Thm:Main}, which states that if, for a generic labeling, the dispersion polynomial  $D(z, \lambda)$  has a linear factor in  $\lambda$, then $\Gamma$ has a fundamental domain     with a nonempty component of support  $0$. We proceed by induction on the number of vertices in the fundamental domain of $\Gamma$.

For $n=1$, it is clear that  if $D(z,\lambda)$   has a linear factor in $\lambda$, then     the fundamental domain of $\Gamma$ has support $0$. 

Assume that Theorem~\ref{Thm:Main} holds for all $n < k$. Suppose that $n=k$. Suppose that $\Gamma$ has no fundamental domain with a support $0$ component and for a generic labeling, $D(z,\lambda)$ has a linear factor in $\lambda$.  Our goal is to obtain a contradiction.  

Let $W$ be a fundamental domain of $\Gamma$. Applying Theorem~\ref{Thm:2} and possibly relabeling the vertices of $W$, we can assume that $U = W \smallsetminus \{ 1\}$ is such that $D|_U(z,\lambda)$ has a linear factor for a generic labeling. 

By the induction hypothesis, as $|U| < k$, it follows that there exists a fundamental domain of $\Gamma_U$ with a component of support $0$.  Without loss of generality, assume that $U \subset W$ is a fundamental domain of $\Gamma_U$ with the maximal number of vertices contained in components of support $0$ among all fundamental domain of $\Gamma_U$.

Let $\hat{U}$ be the collection of all support $0$ components of $U$. Reordering vertices, we can assume $\hat{U} = \{ 2,\dots, l \}$ for some $l \leq k$.
 
As $\hat{U}$ is made up of components of support $0$, $D|_{\hat{U}}(z,\lambda) = D|_{\hat{U}}(\lambda)$ and \begin{equation} D|_U(z,\lambda) = D|_{\hat{U}}(\lambda)D|_{U\smallsetminus \hat{U}}(z,\lambda).\end{equation}

Under our notation, $\Gamma_{U\smallsetminus \hat{U}} $ has no component with support 0. 
By the induction hypothesis, $D|_{U\smallsetminus \hat{U}}(z,\lambda)$ has no linear factors in $\lambda$.

Notice that $L(z)$ is the matrix
\begin{equation} \label{eq:DecompMatrix} {\tiny \begin{pmatrix}
    L(z)_{1,1} & \begin{matrix} L(z)_{1,2} & L(z)_{1,3} & \dots L(z)_{1,l} \end{matrix} &\begin{matrix} L(z)_{1,l+1} & L(z)_{1,l+2} & \dots L(z)_{1,k} \end{matrix} \\ 
    
    \begin{matrix} L(z)_{2,1} \\ L(z)_{3,1}  \\ \vdots \\ L(z)_{l,1} \end{matrix} & L|_{\hat{U}}(z) & \bf{0}  \\
    
    \begin{matrix} L(z)_{l+1,1} \\ L(z)_{l+2,1}  \\ \vdots \\ L(z)_{k,1} \end{matrix} & \bf{0}
 & L|_{U \smallsetminus \hat{U}}(z) \\
\end{pmatrix}}.
\end{equation}

By~\eqref{eq:DecompMatrix}, we may expand $D(z,\lambda)$ as follows:

\begin{align} \begin{split}\label{eq:Decomp} D(z,\lambda) =& \  (L(z)_{1,1} - \lambda) D|_{\hat{U}}(\lambda)D|_{U\smallsetminus \hat{U}}(z,\lambda) \\
&+  D|_{U\smallsetminus \hat{U}}(z,\lambda) P(z,\lambda) 
+  D|_{\hat{U}}(\lambda) Q(z,\lambda) .
    \end{split}
\end{align}

Recall that $D|_{U\smallsetminus \hat{U}}(z,\lambda)$ has no linear factors in $\lambda$ and that $D|_{\hat{U}}(\lambda)$ is strictly a product of linear factors in $\lambda$.
By Theorem~\ref{Thm:2}, $D|_U(z,\lambda)=D|_{\hat{U}}(\lambda)D|_{U\smallsetminus \hat{U}}(z,\lambda)$ must have the same linear factor as $D(z,\lambda)$ for a generic labeling, and so, by~\eqref{eq:Decomp}, 
$P(z,\lambda)$  and  $D|_{\hat{U}}(\lambda)$ share a common factor that is linear in $\lambda$.

Since $P(z,\lambda)$ and $D|_{\hat{U}}(\lambda)$ share a factor, the resultant of these two polynomials must vanish for all $z \in (\CC^*)^d$. This implies that for all the labeling variables $e_{(i,j),a}$ and $v_i$, 
\begin{equation}\label{gfeb271}
    Res(P(z,\lambda), D|_{\hat{U}}(\lambda)) = 0.
\end{equation}

Consider the subgraph $\Gamma_{\{1\} \cup \hat{U}}$. Clearly, by \eqref{eq:Decomp}, we have that
\begin{equation}\label{gfeb2710}
D|_{ \{1\} \cup \hat{U}}(z,\lambda) = L(z)_{1,1} D|_{\hat{U}}(\lambda)
+  P(z,\lambda) .
\end{equation}
Let $H = S(\Gamma_{\{1\} \cup \hat{U}} / \ZZ^d)$.

{\bf Case 1: }$H$ is not connected.
In this case, $\Gamma_{\{1\}}$ must be disconnected with one of the components of $\Gamma_{ \hat{U}}$. This immediately implies that $\Gamma$ has a component of support $0$. Since, by our assumption,  $\Gamma$ has no fundamental domain with a support $0$ component,  we arrive at a contradiction.
  
  {\bf Case 2: }$H$ is connected.
In this case, we first note that $H$ must contain an $l$-vertex subgraph $H'$ with only cut edges (we obtain $H'$ by removing some edges from $H$). 
For some $i \neq j$, 
deleting all edges of the form $(i+b, j+c)$ in $\Gamma_{\{1\} \cup \hat{U}}$, for all $b,c \in \ZZ^d$,   is equivalent to
setting $e_{(i,j),a}$ to $0$ for all $a \in \ZZ^d$,    and is also equivalent to deleting the edge $(i,j)$ in $H$. 
By setting enough \( e_{(i,j),a} \) (for all $a\in \ZZ^d$) to 0, \( \Gamma_{\{1\} \cup \hat{U}} \) reduces to a subgraph \( \Gamma' \), and \( H \) reduces to a subgraph \( H' \), where \( H' \) is a connected graph with  only   cut edges. Let $L'(z)$ be the Floquet matrix of the corresponding   periodic operator on $\Gamma'$. 

Under this specialization of the edge labels, by \eqref{gfeb271} and \eqref{gfeb2710}, we have that  
\begin{equation}~\label{eq:equivRes}
     Res(\det(L'(z)-I\lambda),\det(L'|_{\hat{U}}(z)-I\lambda))=0.
\end{equation}

 As $\hat{U} \subset \{1\} \cup \hat{U}$ has support $0$ and $H'$ has only cut edges, by Theorem~\ref{Lem:CutGraph},  for a generic potential,
 \begin{equation}~\label{eq:final} Res(\det(L'(z)-I\lambda),\det(L'|_{\hat{U}}(z)-I\lambda)) \neq 0.\end{equation}

By   \eqref{eq:equivRes}  and \eqref{eq:final},  we have a contradiction. 

Since we obtain a contradiction in both cases, it follows that  when $n=k$, if, 
for a generic labeling, the dispersion polynomial  $D(z, \lambda)$  has a linear factor in  $\lambda$, then $\Gamma$ has a fundamental domain     with a nonempty component of support  $0$. By induction, we finish the proof.
\end{proof}

\section*{Acknowledgements}
W. Liu extends gratitude to  Mostafa Sabri for his valuable comments on an earlier version of our preprint regarding the historical study of flat bands.
M. Faust thanks Frank Sottile for many helpful conversations on resultants.
W. Liu was a 2024-2025 Simons fellow. This research was partially supported by NSF grants  DMS-2201005, DMS-2052572, DMS-2246031, and DMS-2052519.

\section*{Statements and Declarations}
{\bf Conflict of Interest} 
The authors declare no conflicts of interest.

\vspace{0.2in}
{\bf Data Availability}
Data sharing is not applicable to this article as no new data were created or analyzed in this study.

\def\cprime{$'$}

\bibliographystyle{amsplain}
\bibliography{main}

\end{document}